\documentclass[runningheads]{llncs}
\usepackage{graphicx}
\usepackage{hyperref}

\usepackage{amssymb}
\usepackage{amsmath}
\usepackage{array}%for \newcolumntype macro
\usepackage{stmaryrd}%for \llbracket macro
\usepackage{tikz}
\usepackage[all]{xy}
\usepackage{charter}

\renewenvironment{proof}{{\par\noindent\itshape Proof.}\ \rmfamily}{\hfill\color{gray}$\blacksquare$\par\noindent}

\newcommand{\D}{\mathcal{D}}
\newcommand{\G}{\mathcal{A}}
\newcommand{\K}{\mathcal{K}}
\newcommand{\M}{\mathcal{M}}
\renewcommand{\P}{\mathfrak{P}}
\newcommand{\Q}{\mathcal{Q}}
\newcommand{\R}{\mathfrak{G}}

\renewcommand{\a}{{\overline{a}}}
\renewcommand{\b}{{\overline{b}}}
\renewcommand{\c}{{\overline{c}}}
\newcommand{\p}{{\overline{p}}}
\newcommand{\q}{{\overline{q}}}
\renewcommand{\r}{{\overline{r}}}

\newcommand{\lr}[1]{\langle #1 \rangle}

\newcommand{\BP}{\mathbb{P}}
\newcommand{\BV}{\mathbb{V}}
\newcommand{\EDL}{\textbf{EDL}}
\newcommand{\EDG}{\textbf{EDG}}

\newcolumntype{L}{>{$}l<{$}}
\newcolumntype{R}{>{$}r<{$}}

\renewcommand{\phi}{\varphi}

\begin{document}
\title{Epistemic Logic with Partial Dependency Operator
%\thanks{Supported by organization x.}
}
\author{Xinyu Wang\inst{1}\orcidID{0000-0002-4811-693X}}
\authorrunning{X. Wang}
\institute{School of Electronics Engineering and Computer Science, Peking University\\
\email{xinyuwang1998@pku.edu.cn}}
%Princeton University, Princeton NJ 08544, USA \and
%Springer Heidelberg, Tiergartenstr. 17, 69121 Heidelberg, Germany
\maketitle              % typeset the header of the contribution
\begin{abstract}
In this paper, we introduce \textit{partial} dependency modality $\D$ into epistemic logic so as to reason about \textit{partial} dependency relationship in Kripke models. The resulted dependence epistemic logic possesses decent expressivity and beautiful properties. Several interesting examples are provided, which highlight this logic's practical usage. The logic's bisimulation is then discussed, and we give a sound and strongly complete axiomatization for a sub-language of the logic.

\end{abstract}
\section{Introduction}

Following some previous fundamental work on ``knowing value'' \cite{Wang1,Wang2,Wang3,Wang4,Ding0}, recent years have seen an abundance of interest in this novel kind of non-standard epistemic logic. There has been epistemic logic with functional dependency operator \cite{Ding}, which can help us reason about knowing that the value of certain variable is functionally decided by some other variables. For instance, the agent knows that $y=x^2$, so he knows that $y$ functionally depends on $x$ even if without knowing the exact values of $x$ or $y$.

Nevertheless, the real world is never so ideal as a simple parabola. As a matter of fact, in a lot of practical cases, the value of a dependent variable $y$ is usually influenced by thousands of independent factors as $x_1$, $x_2$, $\ldots$ in a quite complicated way, such that it is virtually impossible to obtain a detailed function to precisely determine the value of $y$. Therefore, in both scientific and social study, the method of control variable gets widely used. We often set the values of all the other variables rigid, only change the value of an independent variable $x$ and observe the change of the dependent variable $y$. If the value of $y$ varies with the value of $x$, then we conclude that $y$ \textit{partially} depends on $x$. In this paper, we introduce modality $\D$ in order to express this kind of \textit{partial} dependency relationship.

There have also been dependence and independence logics dealing with dependency relationship between variables \cite{Vaananen,Gradel,Galliani,Stanford}, and we will discuss our logic's connection to them in Remark \ref{conre}. A similar definition for dependency relationship also appears in Halpern's recent book, pp. 14-19. \cite{Halpern} However, the start point of our work is epistemic logic as well as the Kripke model, and we would like to incorporate \textit{partial} dependency relationship between variables into the agent's knowledge so that we shall obtain an epistemic logic of ``knowing dependency'', which is hence named as dependence epistemic logic. This dependence epistemic logic proves to possess further affluent expressivity as well as rather straightforward properties.

In the Kripke model for our dependence epistemic logic, besides a usual $\sim_i$ S5 equivalence relation representing the agent's knowledge, i.e., all the possible worlds that the agent cannot distinguish, there also exists another $\approx$ S5 equivalence relation representing the physical probability, i.e., all the possible worlds that share the same set of physical laws with the current world. Generally speaking, these two equivalence relations do not have to have any correlation, and thus in the language, the former is characterized by an S5 modality $\K$, while the latter is characterized by another independent S5 modality $\G$. This kind of framework is first introduced by another recent work \cite{Knowall}, and so readers who get confused with the conception of two independent equivalence relations in the model are \textit{strongly} recommended to refer to that paper.

Then the \textit{partial} dependency relationship is valuated in the $\approx$ equivalence class, since dependency relationship between variables is in fact related to some universal physical law and thus concerns not only the current exact world but also all the other worlds that are physically potentially possible. Actually, we introduce two different modalities $\D_g$ and $\D_l$ to characterize \textit{partial} dependency relationship. Their respective semantics is both based on the discussion in the beginning about what modality $\D$ should be like, except for that, the former $\D_g$ is valuated globally in a whole $\approx$ equivalence class, while the latter $\D_l$ fixes one reference point as the current exact world and so is valuated locally. Readers will soon become clear about what $\D_g$ and $\D_l$ mean respectively through the following Section \ref{preli} on preliminaries including the language, model and semantics, and the correlation between these two modalities also gets discussed in Remark \ref{expre}. Examples in Section \ref{examp} illustrate that $\D_g$ is helpful in analyzing universal physical laws while $\D_l$ is useful in expressing counterfactual assumptions, in surprising accordance with our very intuition as well as commonsense, so the practicality of $\D_g$ and $\D_l$ counts to why we introduce both modalities.

The rest of the paper is organized as follows. We lay out the basics of the language and the semantics in Section \ref{preli}. Several interesting examples are illustrated in Section \ref{examp}. A bisimulation notion for this dependence epistemic logic then gets thoroughly discussed in Section \ref{bisim}, followed by a sound and strongly complete axiomatization for a sub-language in Section \ref{axiom}. We finally conclude this paper and propose future research directions in Section \ref{concl}.

\section{Preliminaries}\label{preli}

\begin{definition}[Language $\EDL$]
For a fixed countable set of propositions $\BP$, and a fixed countable set of variables $\BV$, the language $\EDL$ of dependence epistemic logic is defined recursively as:

\begin{align*}
  \phi::=\top\mid p\mid\neg\phi\mid(\phi\land\phi)\mid\K\phi\mid\G\phi\mid\D_g(X,Y)\mid\D_l(X,Y)
\end{align*}

where $p\in\BP$, and $X$ as well as $Y$ are finite subsets of $\BV$. $\D_g(X,Y)$ reads as $Y$ depends on $X$ globally, while $\D_l(X,Y)$ reads as $Y$ depends on $X$ locally. We define $\bot$, $\lor$ and $\to$ as usual.
\end{definition}
\paragraph{\textbf{Important Notation}}In the following parts of this paper, when some property applies to both $\D_g$ and $\D_l$, we will simply omit the subscript and write down only one theorem, lemma, axiom, etc. concerning $\D$ for convenience, and the omitting is also similar for other notations derived from $\D$.

If $X=\{x\}$, we will also denote $\D(\{x\},Y)$ as $\D(x,Y)$ for simplicity, and likely for $Y$ if $Y=\{y\}$.

\begin{definition}[Model]
A dependence epistemic model $\M$ is $\lr{S,T,V,U,\sim_i,\approx}$:

\begin{itemize}
  \item $S$ is a set of possible worlds.
  \item $T$: $S\times\BP\to\{0,1\}$.
  \item $V\supseteq\BV$ is a countable set of variable objects.
  \item $U$: $S\times V\to\mathbb{N}$.
  \item $\sim_i$ is an equivalence relation over $S$.
  \item $\approx$ is an equivalence relation over $S$.
\end{itemize}

As the convention in first-order logic, while $\BV$ in the language are names for variables, $V$ in the model interpret each name with a concrete object and also may consist of other variable objects whose names are not included in the language. Since the language $\EDL$ excludes the equal sign $=$, every name in $\BV$ can be managed to be interpreted differently in $V$, so we simply let $V\supseteq\BV$ and do not make explicit distinctions between names and objects in the following without causing any confusion. Then $U$ is the function that assigns each variable on each possible world with a (countably possible) value, which is supposed to be uniformly numbered by $\mathbb{N}$ for convenience.

Sometimes we apply another extra stipulation on the model in order to satisfy our practical needs: for any proposition $p\in\BP$, it may have its corresponding variable $\p\in\BV$. If so, we then stipulate that $\forall s\in S$, $U(s,\p)=T(s,p)$. The following Subsections \ref{adoor} and \ref{judge} present examples of this kind.
\end{definition}

\begin{definition}[Semantics]\label{seman}
We define that $\forall s,t\in S$, $\forall$ subset $X\subseteq V$, $X_s=X_t$ iff $\forall x\in X,U(s,x)=U(t,x)$, while of course, $X_s\neq X_t$ iff $\exists x\in X,U(s,x)\neq U(t,x)$. A pointed model $\M,s$ is a dependence epistemic model $\M$ with a possible world $s\in S$.

\begin{center}
  \begin{tabular}{|LLL|}
  \hline
  \M,s\vDash\top & \iff & \text{always}\\
  \M,s\vDash p & \iff & T(s,p)=1\\
  \M,s\vDash\neg\phi & \iff & \text{not }\M,s\vDash\phi\\
  \M,s\vDash(\phi\land\psi) & \iff & \M,s\vDash\phi\text{ and }\M,s\vDash\psi\\
  \M,s\vDash\K\phi & \iff & \forall t\in S,t\sim_i s,\M,t\vDash\phi\\
  \M,s\vDash\G\phi & \iff & \forall t\in S,t\approx s,\M,t\vDash\phi\\
  \M,s\vDash\D_g(X,Y) & \iff & \exists u,v\in S,u\approx v\approx s,\\
  & & (V\backslash(X\cup Y))_u=(V\backslash(X\cup Y))_v,X_u\neq X_v,Y_u\neq Y_v\\
  \M,s\vDash\D_l(X,Y) & \iff & \exists t\in S,t\approx s,\\
  & & (V\backslash(X\cup Y))_t=(V\backslash(X\cup Y))_s,X_t\neq X_s,Y_t\neq Y_s\\
  \hline
  \end{tabular}
\end{center}

When it is not that $\M,s\vDash\phi$, we denote it as $\M,s\nvDash\phi$.
\end{definition}

\begin{remark}[Expressivity of $\D_g$ and $\D_l$]\label{expre}
We are able to perceive through Definition \ref{seman} that $\D_g$ is actually definable using $\neg$, $\G$ and $\D_l$, demonstrated as the following:

\begin{align*}
  \D_g(X,Y)\leftrightarrow\neg\G\neg\D_l(X,Y)
\end{align*}

In fact, $\D_l$ is strictly more expressive than $\D_g$, which will become clear to readers through our discussion for bisimulation in Section \ref{bisim}. Nevertheless, due to $\D_g$'s simplicity and usefulness, we will take the language with $\D_g$ but without $\D_l$ as a sub-language of $\EDL$.
\end{remark}

\begin{definition}[Language $\EDG$]
For a fixed countable set of propositions $\BP$, and a fixed countable set of variables $\BV$, the language $\EDG$ is defined recursively as:

\begin{align*}
  \phi::=\top\mid p\mid\neg\phi\mid(\phi\land\phi)\mid\K\phi\mid\G\phi\mid\D_g(X,Y)
\end{align*}

where $p\in\BP$, and $X$ as well as $Y$ are finite subsets of $\BV$.
\end{definition}

The model and semantics are the same.

\begin{remark}[Connection to Independence Logic]\label{conre}
If the total set of variables $V$ is finite and explicitly known, then modality $\D_g$ can be expressed in inclusion logic, a sub-language of independence logic \cite{Inclusion}, as the following:\footnote{As for the notation, we prefer to use $X$ and $Y$ instead of $\overrightarrow{x}$ or $\overrightarrow{y}$. Anyway, their respective meanings in this specific context should be clear to readers.}

\begin{align*}
  \D_g(X,Y)\iff\exists\overrightarrow{w_1}\overrightarrow{x_1}\overrightarrow{y_1}\exists\overrightarrow{w_2}\overrightarrow{x_2}\overrightarrow{y_2}(\overrightarrow{w_1}\overrightarrow{x_1}\overrightarrow{y_1}\subseteq(V\backslash(X\cup Y))XY\land\\
  \overrightarrow{w_2}\overrightarrow{x_2}\overrightarrow{y_2}\subseteq(V\backslash(X\cup Y))XY\land\overrightarrow{w_1}=\overrightarrow{w_2}\land\neg\overrightarrow{x_1}=\overrightarrow{x_2}\land\neg\overrightarrow{y_1}=\overrightarrow{y_2})
\end{align*}

However, this form puts too many restrictions and becomes too lengthy, while we actually want the total set $V$ to be clear from our language so that we can reason with simple and compact logic. In fact, the team model on which independence logic is based is quite different from the Kripke possible world model \cite{Hodges}, both in technique and in philosophical explanation, and hence they are very unlike logics. While independence logic, inherited from first order logic, always reasons globally, epistemic logic, rooted from modal logic, usually reasons locally, which is demonstrated by this obvious fact that local modality $\D_l$ can surely not be defined in independence logic.
\end{remark}

\section{Examples}\label{examp}

\subsection{An Open Door}\label{adoor}

Let $p$ denote that the door of the room is open now, $q$ denote that the agent possesses the key of the door, and $r$ denote that the agent is able to enter the room. Let us suppose that the agent has perfect knowledge, so $\sim_i$ relation is only reflexive. Then we have:\footnote{When drawing all these figures in this paper, for brevity we will omit some relation lines which can be deduced from S5 equivalence class requirements.}

$$\xymatrix{
\txt{$s:p,q,r$\\$\p=1$\\$\q=1$\\$\r=1$}\ar@{-}[r]|{\approx} &
\txt{$p,\neg q,r$\\$\p=1$\\$\q=0$\\$\r=1$}\ar@{-}[r]|{\approx} &
\txt{$\neg p,q,r$\\$\p=0$\\$\q=1$\\$\r=1$}\ar@{-}[r]|{\approx} &
\txt{$\neg p,\neg q,\neg r$\\$\p=0$\\$\q=0$\\$\r=0$}\\
}$$

It is not difficult to observe that $\M,s\vDash\K\D_g(\p,\r)$ and $\M,s\vDash\K\neg\D_l(\p,\r)$. The former says that the agent knows whether he is able to enter the room is somewhat related to whether the door is open now -- if he did not possess the key. And the latter says that under the present situation, since the agent does possess the key, he surely knows that if this precondition is kept unchanged, then he was still able to open the door to enter the room even if the door was now closed. Namely, whether he is able to enter the room does not depend on whether the door is open now, which provides us with a fancy way to express counterfactual assumptions.

\subsection{Error-included Experiment}

Suppose we are carrying out an experiment, and we know from theory that there are two independent variables $x$ and $y$ which may influence the value of the dependent variable $z$, where the value of $x$ is well under control but $y$ represents some random experimental error, and so of course, we cannot control or even measure the value of $y$. The only thing we know about $y$ is that it will be either $1$ or $2$ during every experiment.

Now we have done this experiment twice. When $x = 1$, $z = 1$. When $x = 2$, $z = 2$. By combining all kinds of possibilities, we can have the model as:

$$\xymatrix{
\txt{$x=1$\\$y=1$\\$z=1$}\ar@{-}[r]|{i}\ar@{-}[d]|{\approx} &
\txt{$x=1$\\$y=1$\\$z=1$}\ar@{-}[r]|{i}\ar@{-}[d]|{\approx} &
\txt{$x=1$\\$y=2$\\$z=1$}\ar@{-}[r]|{i}\ar@{-}[d]|{\approx} &
\txt{$x=1$\\$y=2$\\$z=1$}\ar@{-}[d]|{\approx}\\
\txt{$x=2$\\$y=1$\\$z=2$}\ar@{-}[r]|{i} &
\txt{$x=2$\\$y=2$\\$z=2$}\ar@{-}[r]|{i} &
\txt{$x=2$\\$y=1$\\$z=2$}\ar@{-}[r]|{i} &
\txt{$x=2$\\$y=2$\\$z=2$}\\
}$$

Can we be confident that $z$ depends on $x$? Certainly not, because the change of $z$ may be brought about by the change of $y$. As a matter of fact, on every possible world $s$ there is $\M,s\nvDash\K\D_g(x,z)$.

However, if we have further done the third experiment, and when $x = 3$, $z = 3$. Now can we be confident that $z$ depends on $x$? Indeed we can. This fact can be easily observed through the following huge model, where $\M,s\vDash\K\D_g(x,z)$ on every possible world $s$:

$$\xymatrix@C=1pc{
\txt{$x=1$\\$y=1$\\$z=1$}\ar@{-}[d]|{\approx}\ar@{-}[r]|{i} &
\txt{$x=1$\\$y=1$\\$z=1$}\ar@{-}[d]|{\approx}\ar@{-}[r]|{i} &
\txt{$x=1$\\$y=1$\\$z=1$}\ar@{-}[d]|{\approx}\ar@{-}[r]|{i} &
\txt{$x=1$\\$y=1$\\$z=1$}\ar@{-}[d]|{\approx}\ar@{-}[r]|{i} &
\txt{$x=1$\\$y=2$\\$z=1$}\ar@{-}[d]|{\approx}\ar@{-}[r]|{i} &
\txt{$x=1$\\$y=2$\\$z=1$}\ar@{-}[d]|{\approx}\ar@{-}[r]|{i} &
\txt{$x=1$\\$y=2$\\$z=1$}\ar@{-}[d]|{\approx}\ar@{-}[r]|{i} &
\txt{$x=1$\\$y=2$\\$z=1$}\ar@{-}[d]|{\approx}\\
\txt{$x=2$\\$y=1$\\$z=2$}\ar@{-}[d]|{\approx}\ar@{-}[r]|{i} &
\txt{$x=2$\\$y=1$\\$z=2$}\ar@{-}[d]|{\approx}\ar@{-}[r]|{i} &
\txt{$x=2$\\$y=2$\\$z=2$}\ar@{-}[d]|{\approx}\ar@{-}[r]|{i} &
\txt{$x=2$\\$y=2$\\$z=2$}\ar@{-}[d]|{\approx}\ar@{-}[r]|{i} &
\txt{$x=2$\\$y=1$\\$z=2$}\ar@{-}[d]|{\approx}\ar@{-}[r]|{i} &
\txt{$x=2$\\$y=1$\\$z=2$}\ar@{-}[d]|{\approx}\ar@{-}[r]|{i} &
\txt{$x=2$\\$y=2$\\$z=2$}\ar@{-}[d]|{\approx}\ar@{-}[r]|{i} &
\txt{$x=2$\\$y=2$\\$z=2$}\ar@{-}[d]|{\approx}\\
\txt{$x=3$\\$y=1$\\$z=3$}\ar@{-}[r]|{i} &
\txt{$x=3$\\$y=2$\\$z=3$}\ar@{-}[r]|{i} &
\txt{$x=3$\\$y=1$\\$z=3$}\ar@{-}[r]|{i} &
\txt{$x=3$\\$y=2$\\$z=3$}\ar@{-}[r]|{i} &
\txt{$x=3$\\$y=1$\\$z=3$}\ar@{-}[r]|{i} &
\txt{$x=3$\\$y=2$\\$z=3$}\ar@{-}[r]|{i} &
\txt{$x=3$\\$y=1$\\$z=3$}\ar@{-}[r]|{i} &
\txt{$x=3$\\$y=2$\\$z=3$}\\
}$$

Whatever values $y$ may be in the three experiments, there must be at least two experiments in which $y$ is the same, so we can only explain the difference between $z$ in these two experiments as caused by the difference between the value of $x$. This scenario clearly explains why in all the natural science experiments, despite the universal existence of errors, we can still manage to obtain useful conclusions concerning our interested variables, by multiple experiments with relatively large data range.

\subsection{Judging a Case}\label{judge}

We have seen that global modality $\D_g$ can help us analyze complicated experimental results, while local modality $\D_l$ is very helpful in expressing counterfactual assumptions. And there are still trickier things worth examining. Until now, we have only proposed examples including modality $\D$ affecting solely on singletons. It may seem by intuitive guess that $\D(\{a,b\},c)$ tells very similar thing as $\D(a,c)\lor\D(b,c)$. Nevertheless, these two expressions are not exactly the same, and in fact, they may result in quite opposite epistemic consequences, as demonstrated by the following scenario.

Unfortunately, Charles got killed in a tragedy ($c$), which was related to Alan having done something ($a$) and/or Bob having done something ($b$). Firstly, let us suppose that either $a$ or $b$ could happen so as to cause $c$, and only one of them could have happened to be $c$'s indeed cause. However, on the current world $s$ we are yet not sure whether $a$ or $b$ actually happened to be the exact cause of $c$. This can be modeled as the following:

$$\xymatrix@C=4pc{
\txt{$s:a,\neg b,c$\\$\a=1$\\$\b=0$\\$\c=1$}\ar@{-}[r]|{i,\approx} &
\txt{$\neg a,b,c$\\$\a=0$\\$\b=1$\\$\c=1$}\ar@{-}[r]|{\approx} &
\txt{$\neg a,\neg b,\neg c$\\$\a=0$\\$\b=0$\\$\c=0$}\\
}$$

It is not difficult to observe that $\M,s\vDash\K\D_l(\{\a,\b\},\c)\land\K(\D_l(\a,\c)\lor\D_l(\b,\c))$. This is to say, it is within our knowledge that not only the whole group event $\{a,b\}$ is related to $c$, but also either $a$ or $b$ itself is alone related to $c$, namely, their influences on $c$ can be separated in concept. Hence, unless we obtain further evidence to pin down our knowledge in order to determine whether Alan or Bob was the real criminal, by presumption of innocence neither of them can be sentenced guilty for Charles' death.

Now let us turn to a second phenomenon, where $b$'s happening was a direct consequence of $a$'s happening. For instance, let $b$ denote that Bob killed Charles, and $a$ denote that Alan compelled Bob to kill Charles, either by threatening that he would have killed Bob otherwise or by Alan's mind control over Bob through magic or science fiction. In other words, we restrict ourselves to only consider possible worlds on which $a\rightarrow b$ holds in our Kripke model. Under this circumstance, we can model our knowledge as the following:

$$\xymatrix@C=4pc{
\txt{$s:a,b,c$\\$\a=1$\\$\b=1$\\$\c=1$}\ar@{-}[r]|{\approx} &
\txt{$\neg a,b,c$\\$\a=0$\\$\b=1$\\$\c=1$}\ar@{-}[r]|{\approx} &
\txt{$\neg a,\neg b,\neg c$\\$\a=0$\\$\b=0$\\$\c=0$}\\
}$$

At present, even physically speaking $b$ should be the only direct cause of $c$, which is demonstrated by $\G(b\leftrightarrow c)$ holding throughout the model, to our little surprise $\K\D_l(\b,\c)$ does not hold on the current world $s$. As a matter of fact, we have $\M,s\vDash\K\D_l(\{\a,\b\},\c)\land\K(\neg\D_l(\a,\c)\land\neg\D_l(\b,\c))$, a direct contrast against the former scene. This time we not only know that $c$ locally depends on $\{a,b\}$ as a whole, but also know that this dependency relationship should be viewed as an entirety instead of conceptually separable, and therefore, both Alan and Bob should be responsible for Charles' death. Further considering that $\K\G(a\rightarrow b)$ holds on $s$, a legal and rational sentence ought to be that Alan is the principal criminal while Bob is the coerced criminal, which precisely captures the meanings of all the formulae mentioned above.

\section{Bisimulation}\label{bisim}

\begin{definition}[$\Delta(u,v)$]
For any two possible worlds $u,v\in S$, we define:

\begin{align*}
  \Delta(u,v)=\left\{
  \begin{array}{cc}
       \{x\mid x\in\BV,U(u,x)\neq U(v,x)\}, & \text{ if }(V\backslash\BV)_u=(V\backslash\BV)_v\\
       \emptyset, & \text{otherwise}\\
  \end{array}
  \right.
\end{align*}
\end{definition}

\begin{definition}[Evidence]
For any three sets $W$, $X$ and $Y$, $W$ is called an evidence of $\lr{X,Y}$, iff $W\cap X\neq\emptyset$, $W\cap Y\neq\emptyset$, and $W\subseteq X\cup Y$.
\end{definition}

Compared with the original semantics defined in Definition \ref{seman}, we manage to rewrite part of it in an equivalent form as the following:

\begin{lemma}[Evidence Lemma I]\label{ev1le}
\begin{center}
\begin{tabular}{LLL}
  \M,s\vDash\D_g(X,Y) & \iff & \exists u,v\in S,u\approx v\approx s,\Delta(u,v)\text{ is an evidence of }\lr{X,Y}\\
  \M,s\vDash\D_l(X,Y) & \iff & \exists t\in S,t\approx s,\Delta(t,s)\text{ is an evidence of }\lr{X,Y}\\
\end{tabular}
\end{center}
\end{lemma}

\begin{proof}
Directly from the semantics defined in Definition \ref{seman}.
\end{proof}

\begin{definition}[$\P(s)$]
For any possible world $s\in S$, we define:

\begin{align*}
  \P_g(s)=\{\text{nonempty finite set }\Delta(u,v)\mid u,v\in S,u\approx v\approx s\}\\
  \P_l(s)=\{\text{nonempty finite set }\Delta(t,s)\mid t\in S,t\approx s\}
\end{align*}

It is obvious that $\forall s\in S$, $\P_l(s)\subseteq\P_g(s)\subseteq\{$nonempty finite set $W\mid W\subseteq\BV\}$.
\end{definition}

We again manage to rewrite part of the semantics in another equivalent form as the following, making use of the newly defined $\P(s)$:

\begin{lemma}[Evidence Lemma II]\label{ev2le}
\begin{center}
\begin{tabular}{LLL}
  \M,s\vDash\D_g(X,Y) & \iff & \exists W\in\P_g(s),W\text{ is an evidence of }\lr{X,Y}\\
  \M,s\vDash\D_l(X,Y) & \iff & \exists W\in\P_l(s),W\text{ is an evidence of }\lr{X,Y}\\
\end{tabular}
\end{center}
\end{lemma}

\begin{proof}
By Lemma \ref{ev1le}.
\end{proof}

\begin{definition}[Generative]
$\forall s\in S$, any nonempty finite set $W\subseteq\BV$ is called generative from $\P(s)$, iff for any two finite sets $X,Y\subseteq\BV$, such that $W$ is an evidence of $\lr{X,Y}$, there exists $W'\in\P(s)$, such that $W'$ is also an evidence of $\lr{X,Y}$.
\end{definition}

\begin{theorem}[Equivalence Theorem I]\label{eq1th}
For any two pointed models $\M,s$ and $\M',s'$, they satisfy exactly the same $\D(X,Y)$ formulae for any two finite sets $X,Y\subseteq\BV$ iff:

\begin{itemize}
  \item Zig: $\forall W\in\P(s),W\text{ is generative from }\P(s')$.
  \item Zag: $\forall W\in\P(s'),W\text{ is generative from }\P(s)$.
\end{itemize}
\end{theorem}

\begin{proof}
For the direction from left to right, we first concentrate on the Zig condition. If there exists $W\in\P(s)$, such that $W$ is not generative from $\P(s')$, then by definition, there exist two finite sets $X,Y\subseteq\BV$, such that $W$ is an evidence of $\lr{X,Y}$, but there does not exist $W'\in\P(s')$, such that $W'$ is an evidence of $\lr{X,Y}$. By Lemma \ref{ev2le}, this is equivalent to that $\M,s\vDash\D(X,Y)$ but $\M',s'\nvDash\D(X,Y)$, a contradiction. The Zag condition follows by symmetry.

The other direction can also be verified similarly and easily.
\end{proof}

\begin{definition}[$\R(s)$]
For any possible world $s\in S$, we define:

\begin{align*}
  \R(s)=\{W\mid W\text{ is generative from }\P(s)\}
\end{align*}

It is obvious that $\forall s\in S$, $\R_l(s)\subseteq\R_g(s)\subseteq\{$nonempty finite set $W\mid W\subseteq\BV\}$.
\end{definition}

\begin{theorem}[Equivalence Theorem II]\label{eq2th}
For any two pointed models $\M,s$ and $\M',s'$, they satisfy exactly the same $\D(X,Y)$ formulae for any two finite sets $X,Y\subseteq\BV$ iff $\R(s)=\R(s')$.
\end{theorem}

\begin{proof}
Similar to the proof of Theorem \ref{eq1th}.
\end{proof}

Actually, the set $\R(s)$ is the existent and the only greatest generative set from the original $\P(s)$ while keeping satisfying the same formulae for modality $\D$. Therefore, it is worthwhile investigating what characteristics $\R(s)$ possesses, since it precisely determines the modal property of the pointed model $\M,s$. In the following theorem, we manage to express the generative condition for a nonempty finite set $W$ from $\P(s)$ in several different equivalent forms.

\begin{theorem}[Generative Theorem]\label{genth}
$\forall s\in S$, for any nonempty finite set $W\subseteq\BV$, we define $\Sigma(s,W)=\{W'\mid W'\in\P(s),W'\subseteq W\}$, then:

\begin{tabular}{LL}
  \\
  & W\text{ is generative from }\P(s)\\
  \iff & \bigcup\Sigma(s,W)=W,\forall Z\subset W\text{ such that }Z\neq\emptyset,\\
  & \exists W'\in\Sigma(s,W)\text{ such that }W'\cap Z\neq\emptyset\land W'\cap(W\backslash Z)\neq\emptyset\\
  \iff & \bigcup\Sigma(s,W)=W,\forall\Gamma\subset\Sigma(s,W)\text{ such that }\Gamma\neq\emptyset,\\
  & (\bigcup\Gamma)\cap(\bigcup(\Sigma(s,W)\backslash\Gamma))\neq\emptyset\\
  \iff & \bigcup\Sigma(s,W)=W,\forall W'_1,W'_2\in\Sigma(s,W),\text{ define }\mathcal{R}W'_1W'_2\text{ iff }W'_1\cap W'_2\neq\emptyset,\\
  & \text{then }\forall W'_1,W'_2\in\Sigma(s,W),W'_1\text{ connects to }W'_2\text{ by a chain of }\mathcal{R}\text{ relations}\\
\end{tabular}
\end{theorem}

\begin{proof}
Let us concentrate on the following crucial lemma, from which the proof of this theorem follows not difficultly.
\end{proof}

\begin{lemma}[Generative Lemma]\label{genle}
$\forall s\in S$, for any nonempty finite set $W\subseteq\BV$, $W$ is generative from $\P(s)$ iff:

\begin{itemize}
  \item if $|W|=1$, then $W\in\P(s)$.
  \item if $|W|\geqslant 2$, then $\forall Z\subset W$ such that $Z\neq\emptyset$, $\exists W'\in\P(s)$ such that $W'$ is an evidence of $\lr{Z,W\backslash Z}$.
\end{itemize}
\end{lemma}

\begin{proof}
The direction from left to right is immediate. For the direction from right to left, we only have to make use of one simple fact about evidence:

\begin{itemize}
    \item If $W$ is an evidence of $\lr{X,Y}$ and $X\subseteq X'$, then $W$ is an evidence of $\lr{X',Y}$.
\end{itemize}

which, as a matter of fact, can be correspondingly written into a sound axiom regarding modality $\D$:

\begin{align*}
  \D(X,Y)\to\D(X',Y)\text{, given }X\subseteq X'\text{ (Weakening Rule)}\\
\end{align*}

Full axiomatization will later be discussed in the following Section \ref{axiom}.
\end{proof}

The last equivalent condition in Theorem \ref{genth} is to say, we can construct an undirected graph over $\P(s)$ by its elements' intersection relation, and all the generative sets are exactly union of some connected nonempty finite subgraph. This provides us with a clear picture and an intuitive understanding about where every generative set comes from and what $\R(s)$ looks like. Hence given $\P(s)$, there is an explicit algorithm to calculate all the generative nonempty finite sets $W\subseteq\BV$ so as to obtain $\R(s)$.

Finally, taking into account all the modalities including $\K$, $\G$, $\D_g$ and $\D_l$, we are able to define the full bisimulation relation between two models $\M$ and $\M'$:

\begin{definition}[Bisimulation]
A nonempty binary relation $B\subseteq S\times S'$ is called a bisimulation between two models $\M$ and $\M'$ iff:

\begin{itemize}
  \item If $sBs'$, then $\forall p\in\BP$, $T(s,p)=T(s',p)$.
  \item If $sBs'$, then $\R_g(s)=\R_g(s')$.
  \item If $sBs'$, then $\R_l(s)=\R_l(s')$.
  \item Zig for $\K$: if $sBs'$ and $s\sim_i t$, then $\exists t'\in S'$ such that $tBt'$ and $s'\sim_i t'$.
  \item Zig for $\G$: if $sBs'$ and $s\approx t$, then $\exists t'\in S'$ such that $tBt'$ and $s'\approx t'$.
  \item Zag for $\K$: if $sBs'$ and $s'\sim_i t'$, then $\exists t\in S$ such that $tBt'$ and $s\sim_i t$.
  \item Zag for $\G$: if $sBs'$ and $s'\approx t'$, then $\exists t\in S$ such that $tBt'$ and $s\approx t$.
\end{itemize}

When $B$ is a bisimulation between two models $\M$ and $\M'$, we write $B:\M\ \underline{\leftrightarrow}\ \M'$. Furthermore if $sBs'$, we write $B:\M,s\ \underline{\leftrightarrow}\ \M',s'$. If there is a bisimulation $B$ such that $B:\M,s\ \underline{\leftrightarrow}\ \M',s'$, we write $\M,s\ \underline{\leftrightarrow}\ \M',s'$.

We write $\M,s\leftrightsquigarrow\M',s'$, when for any $\EDL$-formula $\phi$, $\M,s\vDash\phi$ iff $\M',s'\vDash\phi$.
\end{definition}

\begin{theorem}[Hennessy-Milner Theorem]
For any two m-saturated models $\M$ and $\M'$, $\forall s\in S$, $\forall s'\in S'$, $\M,s\ \underline{\leftrightarrow}\ \M',s'$ iff $\M,s\leftrightsquigarrow\M',s'$.
\end{theorem}

\begin{proof}
See \cite{Blackburn}. The definition of m-saturated models also appears as Definition 2.53 in that book. It is only the cases for modalities $\D_g$ and $\D_l$ that are added, which just follow from Theorem \ref{eq2th}.
\end{proof}

\section{Axiomatization}\label{axiom}

We only provide a sound and strongly complete axiomatization for language $\EDG$. Nevertheless, the same as the assumed routine in this paper, axioms without subscripts attached to $\D$ are sound with respect to both $\D_g$ and $\D_l$.

To start with, we may notice some obviously sound axioms to characterize the properties of modality $\D$:

\begin{enumerate}
  \item $\D(\emptyset,X)\leftrightarrow\bot$ (Empty Set Rule)
  \item $\D(X,Y)\leftrightarrow\D(Y,X)$ (Symmetry Rule)
  \item $\D(X,Y)\to\D(X',Y)$, given $X\subseteq X'$ (Weakening Rule)
  \item $\D(X,Y)\leftrightarrow\D(X\backslash Y,Y)\lor\D(X\cap Y,Y)$ (Separation Rule)
\end{enumerate}

Although these na\"{i}ve axioms indeed look very similar to those in independence logic \cite{Axiom}, pitifully in our dependence epistemic logic, they alone are away from being complete. The good news is that, we can instead find some conciser axioms, which entirely grasp the full properties of modality $\D$ itself, and from which all the above sound axioms can surely be deduced.

For brevity, let us first define an auxiliary notation:

\begin{definition}[$\Q(W)$]
For any nonempty finite set $W\subseteq\BV$, we define:

\begin{align*}
  \Q(W)::=\left\{
  \begin{array}{cc}
    \D(W,W), & |W|=1\\
    \bigwedge\limits_{Z\subset W,Z\neq\emptyset}\D(Z,W\backslash Z), & |W|\geqslant 2\\
  \end{array}
  \right.
\end{align*}
\end{definition}

Recall Lemma \ref{genle}, readers should be aware that this $\Q(W)$ precisely depicts the minimum necessary $\D(X,Y)$ formulae, such that $W$ is an evidence of $\lr{X,Y}$. Taking advantage of this notation, we can write down rather concise sound axioms about modality $\D$ so as to obtain a complete axiomatization, as the following Q and E Axioms for $\D$ in Theorem \ref{axith}:

\begin{theorem}[Axiomatization]\label{axith}
The following proof system is sound and strongly complete with respect to language $\EDG$.

\begin{center}
\begin{tabular}{|R|L|}
\hline
  \text{TAUT} & \text{all instances of tautologies}\\
  \text{MP} & \text{from }\phi\text{ and }\phi\to\psi\text{ infer }\psi\\
  \text{NEC for }\K & \text{from }\phi\text{ infer }\K\phi\\
  \text{DIST for }\K & \K(\phi\to\psi)\to(\K\phi\to\K\psi)\\
  \text{T for }\K & \K\phi\to\phi\\
  \text{4 for }\K & \K\phi\to\K\K\phi\\
  \text{5 for }\K & \neg\K\phi\to\K\neg\K\phi\\
  \text{NEC for }\G & \text{from }\phi\text{ infer }\G\phi\\
  \text{DIST for }\G & \G(\phi\to\psi)\to(\G\phi\to\G\psi)\\
  \text{T for }\G & \G\phi\to\phi\\
  \text{4 for }\G & \G\phi\to\G\G\phi\\
  \text{5 for }\G & \neg\G\phi\to\G\neg\G\phi\\
  \text{Q for }\D & \D(X,Y)\leftrightarrow\bigvee\limits_{X'\subseteq X,Y'\subseteq Y,X',Y'\neq\emptyset}\Q(X'\cup Y')\text{, given }X,Y\neq\emptyset\\
  \text{E for }\D & \D(\emptyset,X)\leftrightarrow\D(X,\emptyset)\leftrightarrow\bot\\
  \text{4 for }\D_g & \D_g(X,Y)\to\G\D_g(X,Y)\\
\hline
\end{tabular}
\end{center}
\end{theorem}

\begin{proof}
We only show completeness. The proof is almost routine, so we concentrate on how the canonical model is built and on the Truth Lemma for modality $\D_g$. Notice that the Axiom of Choice has to be made use of in the proof.
\end{proof}

\begin{definition}[Canonical Model]
For a fixed language with a set of propositions $\BP$ and a set of variables $\BV$, we first expand this language to $\BP^C$ and $\BV^C$, such that $\BP^C=\BP$, $\BV^C\supseteq\BV$, and that $\BV^C$ is countably infinite. Obviously, if an MCS is satisfied in the canonical model of the expanded language, its restriction down to the original language will also be satisfied in the same model.

The canonical dependence epistemic model $\M^C$ is $\lr{S^C,T^C,V^C,U^C,\sim_i^C,\approx^C}$:

\begin{itemize}
  \item $S^C$ is the set of all MCSs.
  \item $T^C:S^C\times\BP^C\to\{0,1\}$. $\forall s\in S^C$, $\forall p\in\BP^C$, $T^C(s,p)=1$ iff $p\in s$.
  \item $V^C=\BV^C$.
  \item $\sim_i^C$ is an equivalence relation over $S^C$. $\forall s,t\in S^C$, $s\sim_i^Ct$ iff $\{\K\phi\mid\K\phi\in s\}=\{\K\phi\mid\K\phi\in t\}$.
  \item $\approx^C$ is an equivalence relation over $S^C$. $\forall s,t\in S^C$, $s\approx^Ct$ iff $\{\G\phi\mid\G\phi\in s\}=\{\G\phi\mid\G\phi\in t\}$.
  \item $U^C:S^C\times V^C\to\mathbb{N}$. For each fixed $\approx^C$ equivalence class $S_\approx\subseteq S^C$, we assign $V^C$'s values on every possible world $s\in S_\approx$ as the following procedure:
\end{itemize}
\end{definition}

By the 4 Axiom for $\D_g$ in Theorem \ref{axith}, it is easy to see that if $s\approx^Ct$, then $\{\D_g(X,Y)\mid\D_g(X,Y)\in s\}$ = $\{\D_g(X,Y)\mid\D_g(X,Y)\in t\}$. So suppose arbitrary $s\in S_\approx$, $W_\approx=\{\text{nonempty finite set }W\mid W\subset\BV^C,\Q_g(W)\in s\}$ is a well defined set, regardless of which possible world $s$ we choose from $S_\approx$.

\begin{claim}
$W_\approx$ is countable. Therefore, we can suppose a well order $<_W\cong\omega$ on it.
\end{claim}

We define a constant function $f_0:V^C\to\mathbb{N}$, $\forall x\in V^C$, $f_0(x)=0$.

\begin{lemma}[Canonical Assignment]
For every $W\in W_\approx$, we can simultaneously find two corresponding functions $f_1^W:V^C\to\mathbb{N}$ and $f_2^W:V^C\to\mathbb{N}$ such that:

\begin{itemize}
  \item $\{f_1^W(x)\neq f_2^W(x)\mid x\in V^C\}=W$;
  \item if $W_1,W_2\in W_\approx$, $W_1\neq W_2$, then $\{f_i^{W_1}(x)\neq f_j^{W_2}(x)\mid x\in V^C\}$ is countably infinite, $i,j\in\{1,2\}$;
  \item $\{f_i^W(x)\neq f_0(x)\mid x\in V^C\}$ is countably infinite, $i\in\{1,2\}$.
\end{itemize}
\end{lemma}

\begin{proof}
Noticing that there are countably infinite variables in $V^C$ which can be assigned to countably infinite values, while $W_\approx$ is also countable and all the sets $W\in W_\approx$ are finite, we are sure that these requirements can be satisfied. For example, we manage to designate $f_1^W$ and $f_2^W$ for every $W\in W_\approx$ one by one, along the well order $<_W$. Since every $W$ is finite, to satisfy the first requirement, the ranges of $f_1^W$ and $f_2^W$ can be controlled to be both finite. For the second requirement, if $W_1<_W W_2$, we let the ranges of $f_i^{W_2}$ and $f_j^{W_1}$ not intersect. For the third requirement, we let $0$ not be in $f_i^W$'s range.
\end{proof}

We collect all these functions as $F_\approx=\{f_i^W\mid W\in W_\approx,i\in\{1,2\}\}\cup\{f_0\}$.

\begin{claim}
$F_\approx$ is countable. Therefore, we can suppose a well order $<_F$ on it.
\end{claim}

Then by the Well-ordering Theorem, we can also suppose a well order $<_S$ on $S_\approx$. By correlating these two well orders $<_F$ and $<_S$, we can use function $f\in F_\approx$ to assign $V^C$'s values on possible world $s\in S_\approx$, such that $\forall x\in V^C$, $U^C(s,x)=f(x)$. As any two well orders can be compared, during this correlating procedure, one and only one of the following three conditions will occur:

\begin{itemize}
  \item If $<_F\cong<_S$, done.
  \item If we first run out of functions from $F_\approx$, then we use $f_0$ to assign $V^C$'s values for all the other left possible worlds in $S_\approx$.
  \item If we first run out of possible worlds from $S_\approx$, then we arbitrarily choose one possible world $s\in S_\approx$, and copy it many times so as to match all the other left functions in $F_\approx$. All these copies of $s$, along with the original one, of course share the same $T^C$, and are in the same $\sim_i^C$ and $\approx^C$ equivalence classes. Obviously, this copy will not cause any unpleasant consequences.
\end{itemize}

\begin{lemma}[Truth Lemma for Modality $\D_g$]
$\forall s\in S^C$, $\forall$ finite subsets $X,Y\subset\BV^C$, $\D_g(X,Y)\in s\iff\M^C,s\vDash\D_g(X,Y)$.
\end{lemma}

\begin{proof}
The cases when $X=\emptyset$ or $Y=\emptyset$ follow immediately from the E Axiom for $\D$ in Theorem \ref{axith}, so we concentrate on the situations when $X\neq\emptyset$ and $Y\neq\emptyset$. By Lemma \ref{ev1le}, $\M^C,s\vDash\D_g(X,Y)\iff\exists u,v\in S^C$, $u\approx^Cv\approx^Cs$, such that $\Delta(u,v)$ is an evidence of $\lr{X,Y}$.

For the direction from right to left, from the above assignment procedure of $U^C$ in the canonical model, we have $\Q_g(\Delta(u,v))\in s$. Since $\Delta(u,v)$ is an evidence of $\lr{X,Y}$, by making use of the Weakening Rule it is not difficult to reason that $\D_g(X,Y)\in s$.

For the direction from left to right, considering the Q Axiom for $\D$ in Theorem \ref{axith}, at least one of the $\Q_g(X'\cup Y')$ in the big disjunction is in $s$, and thus from the above assignment procedure of $U^C$ in the canonical model, $\exists u,v\in S^C$, $u\approx^Cv\approx^Cs$, such that $\Delta(u,v)=X'\cup Y'$. Since $X'\subseteq X,Y'\subseteq Y,X',Y'\neq\emptyset$, obviously $X'\cup Y'$ is just an evidence of $\lr{X,Y}$ and hence $\M^C,s\vDash\D_g(X,Y)$.
\end{proof}

\section{Conclusions and Future Work}\label{concl}

In this paper, we come up with dependence epistemic logic in order to reason about \textit{partial} dependency relationship between variables under an epistemic scenario. Several interesting examples are proposed, which demonstrate our language's affluent expressivity and practical usage. Besides that, the essential properties of the logic are straightforward to understand, and hence we further discuss its bisimulation relation and manage to provide a sound and strongly complete axiomatization system for the simpler sub-language $\EDG$.

Nevertheless, there still remains much work to be done in the future. The axiomatization of the full language $\EDL$ is yet unknown. It will also be helpful to elaborate on other computational properties of this logic, such as decidability. Besides, as we only deal with the presence of a single agent in this paper, extending this dependence epistemic logic to cases with multiple agents may result in more interesting results. Moreover, it seems to be an exciting idea to add other modalities into this framework so that we will be able to reason about knowing dependency, knowing value, knowing how as well as many other epistemic assertions all together.

\section*{Acknowledgements}

The author would like to thank Yanjing Wang for proposing the $\approx$ equivalence relation as well as a lot of other useful comments and suggestions on this paper.

The author would like to thank Fan Yang for coming up with the formula in Remark \ref{conre}.

The author would also like to thank all the teachers and students who have been taking part in the epistemic logic course, from discussions with whom a lot of inspirations have been stimulated.
%\bibliographystyle{splncs04}
%\bibliography{cited}
%

\end{document}